\documentclass[12pt]{amsart}

\usepackage{amsfonts,amssymb,amsmath,amsthm,epsfig,euscript}

\newtheorem{theorem}{Theorem}

\newtheorem{observation}{Observation}

\theoremstyle{definition}
{
\newtheorem{definition}{Definition}

}

\long\def\symbolfootnote[#1]#2{\begingroup
\def\thefootnote{\fnsymbol{footnote}}\footnote[#1]{#2}\endgroup}

\newcommand{\red}{\mathrm{red}}

\def\P{\mathbb{P}}

\title{Existence of $u$-representation of graphs}

\author[S. Kitaev]{Sergey Kitaev}
\thanks{Department of Computer and Information Sciences, University of Strathclyde, Glasgow G1 1XH, UK. Email: \texttt{sergey.kitaev@cis.strath.ac.uk}}

\begin{document}
\maketitle

\begin{abstract}
\noindent
Recently, Jones et al. introduced the study of $u$-representable graphs, where $u$ is a word over $\{1,2\}$ containing at least one 1. The notion of a $u$-representable graph is a far-reaching generalization of the notion of a word-representable graph studied in the literature in a series of papers. 

Jones et al.\ have shown that any graph is $11\cdots 1$-representable assuming that the number of 1s is at least three, while the class of 12-rerpesentable graphs is properly contained in the class of comparability graphs, which, in turn, is properly contained in the class of word-representable graphs corresponding to 11-representable graphs. Further studies in this direction were conducted by Nabawanda, who has shown, in particular, that the class of 112-representable graphs is not included in the class of word-representable graphs. 

Jones et al.\  raised a question on classification of $u$-representable graphs at least for small values of $u$. In this paper we show that if $u$ is of length at least 3 then any graph is $u$-representable. This rather unexpected result shows that from existence of representation point of view there are only two interesting non-equivalent cases in the theory of $u$-representable graphs, namely, those of $u=11$ and $u=12$.\\

\end{abstract}

\section{Introduction}

The notion of a {\em word-representable graph} first appears 
in \cite{KP2008}.  A simple graph $G=(V,E)$ is word-representable if there exists a word $w$ over the alphabet $V$ such that letters $x$ and $y$, $x\neq y$, alternate in $w$ if and only if $xy$ is an edge in $E$. By definition, each letter in $V$ must occurs at least once in $w$. For example, the cycle graph on four vertices labeled by 1, 2, 3 and 4 in clockwise direction can be represented by the word 14213243.  Some graphs are word-representable, the others are not; the minimum non-word-representable graph is the {\em wheel} $W_5$ on 6 vertices. In fact, it is an NP-hard problem to determine whether or not a given graph is word-representable. These graphs have been studied in a series of papers (see~\cite{JKPR} and references therein), and they will be the main subject of an up-coming book~\cite{KL}.

The key observation that motivated \cite{JKPR} by Jones et al.\
was the fact that the study of word-representable graphs
is naturally connected with the study of patterns in
words (see \cite{Kit} for a comprehensive introduction to the respective area). That is, let $\P = \{1,2, \ldots \}$ be the set of positive integers and $\P^*$ be the set of all words over $\P$. 

If $n \in \P$, then
we let $[n] = \{1, \ldots, n\}$ and $[n]^*$ denote the set of
all words over $[n]$. Given a word $w =w_1 \cdots w_n$ in $\P^*$, we let
$A(w)$ be the set of letters occurring in $w$.  For example,
if $w = 641311346$, then $A(w) = \{1,3,4,6\}$.  If
$B \subseteq A(w)$, then we let $w_B$ be the word that results
from $w$ by removing all the letters in $A(w) \setminus B$.  For
example, if  $w = 4513113458$, then $w_{\{1,3,5\}} =
5131135$. If $u \in \P^*$, we let
$\red(u)$ be the word that is obtained from $u$ by replacing
each occurrence of the $i$-th smallest letter by $i$. For example, if $u = 247429$, then
$\red(u) = 123214$.

Given a word $u =u_1 \cdots u_j \in \P^*$ such
that $\red(u) =u$,  we say
that a word $w = w_1 \cdots w_n \in \P^*$ has a {\em $u$-match starting at position $i$} if $\red(w_i w_{i+1} \cdots w_{i+j-1}) =u$. Note that we can now rephrase the definition of a word-representable graph
by saying that a graph $G=([n],E)$ is word-representable
if and only if there is a word $w \in [n]^*$ such that $A(w)=[n]$ and for all $x,y \in [n]$,
$xy \in E$ if and only if $w_{\{x,y\}}$ has no $11$-matches.

The last observation leads  to the following more general definition appearing in \cite{JKPR}.
Given a word $u \in [2]^*$ such that $\red(u) =u$, we
say that a graph $G$ is {\em $u$-representable} if
and only if there is a labeling $G=([n],E)$, and a word $w \in [n]^*$ such
that $A(w)=[n]$ and for all $x,y \in [n]$, $xy \in E$ if and only if
$w_{\{x,y\}}$ has no $u$-matches. In this case we say
that $w$ {\em $u$-represents} $G=([n],E)$. It is easy to see that the class of $u$-representable graphs is hereditary.

Note that requiring existence of a proper labelling of a graph is essential in the last definition, which is a fundamental difference between the theory of word-representable graphs and its generalization, the theory of $u$-representable graphs. Indeed, it can be easily seen that there does not exist a word $12$-representing the path graph $P_3$ labelled by 1, 2, 3 with 1 and 3 corresponding to the leaves, while, say, the word 231 $12$-represents $P_3$ labelled by 1, 2, 3 with 2 and 3 corresponding to the leaves, thus making $P_3$ be $12$-representable. However, the results in this paper will show that labelling is not important when $u$ is of length 3 or more, since in this case any graph with any labelling is $u$-representable. 

Many results are obtained in the literature on 11-representable graphs. In particular, a characterization of these graphs is known in terms of so-called {\em semi-transitive orientations}; see~\cite{KL} and references therein. The class of 12-representable graphs remains less studied, although there is a rich theory behind it~\cite{JKPR}. In particular, one has the following inclusion:
$$\mbox{{\em permutation graphs}}\subset\mbox{12-representable graphs}\subset$$
$$\mbox{{\em comparability graphs}}\subset\mbox{11-representable graphs.}$$
Also, while any tree can be easily 11-represented using two copies of each letter, only those of them that are {\em double caterpillar} are 12-representable. Moreover, while the length of a word 11-representing a graph on $n$ vertices is bounded above by $2n^2$ (with the longest known to us cases being of length $\lfloor n^2/2\rfloor$), the upper bound for shortest words in the 12-representation case is $2n$.  Even though 12-representation deals with shorter lengths of respective words, one of the difficulties is that labelling graphs differently may lead to opposite results.

Given how involved the theories of word-representable (that is, 11-representable) and 12-representable graphs are, it comes as a surprise that every graph is $111$-representable~\cite{JKPR}.  In fact, Jones et al.\ \cite{JKPR} showed that for every $k \geq 3$, any graph is $1^k$-representable, where for a letter $x$, $x^k$ denotes $k$ copies of $x$ concatenated together. One of questions raised by Jones et al.\ is on whether we can classify $u$-representable graphs for small words $u$ such as $u =121$ or $u=112$. To shed some light on this direction, Nabawanda~\cite{Olivia} studied 112-representing graphs, in particular, showing that this class is not included in the class of word-representable graphs.

In this paper we show that if $u$ is of length at least 3 then {\em any} graph is $u$-representable. This rather unexpected result leaves only two interesting non-equivalent cases  in the theory of $u$-representable graphs from existence of representation point of view, namely, the cases of $11$-representable grpahs and $12$-representable graphs. Note that the class of $12$-representable graphs is trivially equivalent, via taking the complement defined in the next section, to the class of $21$-representable graphs. 

This paper is organized as follows. In Section~\ref{preliminaries} we will introduce all necessary definitions and some basic facts about $u$-representable graphs. Also, in Section~\ref{preliminaries} we sketch the proof appearing in~\cite{JKPR} that any graph is $1^k$-representable for $k\geq 3$. In Section~\ref{main-res-sec} we present our main results (Theorems \ref{thm-121}, \ref{thm-112}, \ref{thm-1122}) and in Section~\ref{concl-remarks} we suggest some directions for further research. 

\section{Preliminaries}\label{preliminaries}

A simple graph $G =(V,E)$ consists of a set of vertices
$V$ and a set of edges $E$ of the form $xy$, where
$x,y \in V$ and $x \neq y$.  All graphs considered in this paper are simple and finite, and we assume that $V=[n] = \{1, \ldots, n\}$ for some $n$.

\begin{definition}\label{def-match-repr} Let $u = u_1 \cdots u_k$ be a word in $\{1,2\}^*$ such
that $\red(u) =u$ and $k\geq 2$.  Then we say that a labeled graph
$G =([n],E)$ is {\em $u$-representable}
if there is a word $w \in \P^*$ such that $A(w)=[n]$ and for all $x,y \in [n]$, $xy  \in E$ if and only if
$w_{\{x,y\}}$ has no $u$-match. We say that an
 unlabeled graph $H$ is $u$-representable if there exits
a labeling of $H$, $H'=([n],E')$, such
that $H'$ is $u$-representable. \end{definition}

As is noted in \cite{JKPR}, there are some natural symmetries
among $u$-representable graphs. That
is, suppose that $u =u_1 \cdots u_k \in \P^*$ and
$\red(u) =u$.  Let the {\em reverse} of
$u$ be the word $u^r = u_k u_{k-1} \cdots u_1$.
Then for any word
$w \in \P^*$, it is easy to see that
$w$ has a $u$-match if and only if $w^r$ has a $u^r$-match.
This justifies the following
observation.

\begin{observation}\label{lem:reverse-u}\cite{JKPR}
Let $G =(V,E)$ be a graph and  $u\in\P^*$ be such that $\red(u) =u$. Then
$G$ is $u$-representable if and only if $G$ is
$u^r$-representable.
\end{observation}

For any word
$w =w_1 \cdots w_k \in \P^*$ whose largest letter is $n$,
we define $w$'s {\em complement} $w^c$ to be $(n+1-w_1) \cdots (n+1-w_k)$.
It is easy to see that
$w$ has a $u$-match if and only if $w^c$ has a $u^c$-match.
Given a graph $G =([n],E)$, we let
the {\em supplement} of $G$ be defined by
$\bar{G} =([n],\bar{E})$ where for all $x,y \in [n]$,
$xy \in E$ if and
only if $n+1-x$ and $n+1 -y$ are adjacent in $\bar{G}$.  One can think of
the supplement of the graph $G =(V,E)$ as a relabeling where
one replaces each label $x$ by the label
$n+1-x$.

It is easy to see that if $w$ $u$-represents $G=([n],E)$
then $w^c$ $u^c$-represents $\bar{G}$, which justifies the following
observation.

\begin{observation}\label{lem:complement-u}\cite{JKPR}
Let $G =([n],E)$ be a graph, and
$u$ be a word in $[n]^*$ such
that $\red(u) =u$. Then
$G$ is $u$-representable if and only if $\bar{G}$ is
$u^c$-representable.
\end{observation}

The following observation will be used as the base case in inductive arguments in all our proofs in this paper.  

\begin{observation}\label{observ-Kn} If $G$ is the complete graph $K_n$ on vertex set $[n]$, then $G$ is $u$-representable for every $u$ of length at least $3$ by any permutation of $[n]$, in particular,  $w = 12\cdots n$ $u$-represents $K_n$.\end{observation}

The goal of this paper is to generalize the following theorem appearing in \cite{JKPR} to the case of an arbitrary $u$ of length at least 3. We sketch the proof of the theorem since it gives an idea of our approach to prove similar statements, although in different cases we use different constructions. In what follows, $p(w)$ denotes the {\em initial permutation} of $w$. That is, $p(w)$ is obtained from $w$ by removing all but the leftmost occurrence of each letter. For example,  $p(31443266275887)=31426758$.

\begin{theorem}\label{11111-matching}\cite{JKPR} For any fixed $k \geq 3$,
every graph $G$ is $1^k$-representable.
\end{theorem}

\begin{proof} If $G=K_n$ then by Observation~\ref{observ-Kn} $G$ is $1^k$-representable.

We proceed by induction on the number of edges in a graph with the base case being $K_n$. Our goal is to show that if $G$ is $1^k$-representable, then the graph $G'$ obtained from $G$ by removing any edge
$ij$ is also $1^k$-representable.

Suppose that $w$ $1^k$-represents $G=([n],E)$ and $\pi$ is any permutation of $[n]\backslash\{i,j\}$. Then the word
$$w'=i^{k-1}\pi ip(w)w$$
$1^k$-represents $G'$, which can be checked by considering several cases and to make sure that adding $i^{k-1}\pi ip(w)$ to the left of $w$ results in removing the edge $ij$ but not affecting any other edge/non-edge in $G$. See \cite{JKPR} for further details of the proof.
\end{proof}

\section{$u$-representing graphs for $u$ of length at least 3}\label{main-res-sec}

To proceed, we make a convention that if $ij$ denotes an edge in a graph then $i<j$. Also, for a word $w=w_1w_2\cdots w_k\in\{1,2\}^*$ we let $w[i,j]$ denote the word obtained from $w$ by using the substitution: $1\rightarrow i$ and $2\rightarrow j$ for some letters $i$ and $j$. Finally, recall that for a letter $x$, $x^k$ denotes $k$ copies of $x$ concatenated together.

Note that by Observation~\ref{lem:complement-u}, while studying $u$-representation of graphs, we can assume that $u=1u'$ for some $u'\in\{1,2\}^*$. Taking into account that the case $u=1^k$ is given by Theorem~\ref{11111-matching}, we will only need to consider the following three disjoint cases for $u=1u_2u_3\cdots u_k$, where $k\geq 3$:  
\begin{itemize}
\item $u=1^a2^b1u_{a+b+2}u_{a+b+3}\cdots u_{k}$, where $a,b\geq 1$ (Theorem~\ref{thm-121}); 
\item $u=1^{k-1}2$ (Theorem~\ref{thm-112}); 
\item $u=1^a2^b$, where $a,b\geq 2$ and $a+b=k$ (Theorem~\ref{thm-1122}). 
\end{itemize}
Note that the case of $u=12^{k-1}$ is equivalent to the case of $u=1^{k-1}2$ via Observations~\ref{lem:reverse-u} and~\ref{observ-Kn} (by applying reverse complement to words $u$-representing a graph), and thus it is omitted. In each of the three cases above, we will show that any graph can be $u$-represented.

\begin{theorem}\label{thm-121} Let $u=1^a2^b1u_{a+b+2}u_{a+b+3}\cdots u_{k}$, where $a,b\geq 1$ and $u_{a+b+2}u_{a+b+3}\cdots u_{k}\in\{1,2\}^*$. Then every graph $G$ is $u$-representable.
\end{theorem}

\begin{proof}  
If $G=K_n$ then by Observation~\ref{observ-Kn} $G$ is $u$-representable. We proceed by induction on the number of edges in a graph with the base case being $K_n$. Our goal is to show that if $G$ is $u$-representable, then the graph $G'$ obtained from $G$ by removing any edge $ij$ ($i<j$) is also $u$-representable.

Suppose that $w$ $u$-represents $G$. We claim that the word
$$w'=u[i,j]1^{b+1}2^{b+1}\cdots n^{b+1}w$$
$u$-represents $G'$. 

Indeed, the vertices $i$ and $j$ are not connected any more because $w'_{\{i,j\}}$ contains $u$ (formed by the $k$ leftmost elements of $w'_{\{i,j\}}$). Also, no new edge can be created in $G$ because $w$ is a subword of $w'$. Thus, we only need to show that each edge $ms$ ($m<s$) $u$-represented by $w$ is still $u$-represented by $w'$ if $\{m,s\} \neq \{i,j\}$. We have five cases to consider.

\begin{enumerate}
\item\label{121case=m=i=etc} Suppose $m=i$ and $s\neq j$ $(i<s)$. In this case, $w'_{\{i,s\}} = i^ds^{b+1}w_{\{i,s\}}$, where $d\geq a+b+2$. Because of $s^{b+1}$, there is no $u$-match in $w'_{\{i,s\}}$ that begins to the left of $w_{\{i,s\}}$, and because $w_{\{i,s\}}$ itself has no $u$-matches, $w'$ $u$-represents the edge $is$.
\item\label{121case=s=i=etc} Suppose $s=i$ and $m<i$. In this case, $w'_{\{m,i\}} = i^dm^{b+1}i^{b+1}w_{\{m,i\}}$, where $d\geq a+1$. Because of $i^{b+1}$, there is no $u$-match in $w'_{\{m,i\}}$ that begins to the left of $w_{\{m,i\}}$, and because $w_{\{m,i\}}$ itself has no $u$-matches, $w'$ $u$-represents the edge $mi$.
\item Suppose $m\neq i$ and $s=j$ $(m<j)$. This case is essentially the same as (\ref{121case=s=i=etc}) after the substitution $i\rightarrow j$.
\item Suppose $m=j$ and $s>j$. This case is essentially the same as (\ref{121case=m=i=etc}) after the substitution $i\rightarrow j$.
\item Suppose $m,s \not \in \{i,j\}$. In this case, $w'_{\{m,s\}} = m^{b+1}s^{b+1}w_{\{m,s\}}$. Because of $s^{b+1}$, there is no $u$-match in $w'_{\{m,s\}}$ that begins to the left of $w_{\{m,s\}}$, and because $w_{\{m,s\}}$ itself has no $u$-matches, $w'$ $u$-represents the edge $ms$.
\end{enumerate}

We have proved that $w'$ $u$-represents $G'$, as desired.\end{proof}

\begin{theorem}\label{thm-112} Let $u=1^{k-1}2$, where $k\geq 3$. Then every graph $G$ is $u$-representable.
\end{theorem}

\begin{proof}  
If $G=K_n$ then by Observation~\ref{observ-Kn} $G$ is $u$-representable. We proceed by induction on the number of edges in a graph with the base case being $K_n$. Our goal is to show that if $G$ is $u$-representable, then the graph $G'$ obtained from $G$ by removing any edge $ij$ ($i<j$) is also $u$-representable.

Suppose that $w$ $u$-represents $G$. We claim that the word $w'$ defined as
$$i^{k-2}(i+1)(i+2)\cdots (j-1)(j+1)(j+2)\cdots nij(j+1)\cdots n(i+1)(i+2)\cdots nw$$
$u$-represents $G'$. 

Indeed, the vertices $i$ and $j$ are not connected any more because $w'_{\{i,j\}}$ contains $i^{k-1}j$. Also, no new edge can be created in $G$ because $w$ is a subword of $w'$. Thus, we only need to show that each edge $ms$ ($m<s$) $u$-represented by $w$ is still $u$-represented by $w'$ if $\{m,s\} \neq \{i,j\}$. We have five cases to consider. 

\begin{enumerate}
\item Suppose $m=i$ and $s\neq j$ $(i<s)$. In this case, $w'_{\{i,s\}} = i^{k-2}sissw_{\{i,s\}}$ or $w'_{\{i,s\}} = i^{k-2}sisw_{\{i,s\}}$ depending on whether $s>j$ or not, respectively. Because of the subword $sis$ observed in both cases, there is no $u$-match in $w'_{\{i,s\}}$ that begins to the left of $w_{\{i,s\}}$, and because $w_{\{i,s\}}$ itself has no $u$-matches, $w'$ $u$-represents the edge $is$.

\item\label{subcase-1} Suppose $s=i$ and $m<i$. However, in this case no such $m$ appears to the left of $w$ in $w'$, and thus there is no $u$-match in $w'_{\{m,i\}}$ that begins to the left of $w_{\{m,i\}}$. Since $w_{\{m,i\}}$ has no $u$-match, $w'$ $u$-represents the edge $mi$.

\item Suppose $m\neq i$ and $s=j$ $(m<j)$. In this case, we can assume that $m>i$ since otherwise we have a case similar to (\ref{subcase-1}). We have that $w'_{\{m,j\}} = mjmjw_{\{m,j\}}$. Clearly, no $u$-match can begin at $mjmj$. Since $w_{\{m,j\}}$ has no $u$-match, $w'$ $u$-represents the edge $mj$.

\item Suppose $m=j$ and $s>j$. In this case, we have $w'_{\{j,s\}} = sjsjsw_{\{m,j\}}$. Clearly, no $u$-match can begin at one of the letters in the subword $sjsjs$. Since $w_{\{j,s\}}$ has no $u$-match, $w'$ $u$-represents the edge $js$.

\item Suppose $m,s \not \in \{i,j\}$. In this case, we can assume that $m>i$ since otherwise we have a case similar to (\ref{subcase-1}). We have three subcases to consider here.
\begin{itemize}
\item If $i<m<s<j$ then $w'_{\{m,s\}} = msmsw_{\{m,s\}}$.
\item If $i<m<j<s$ then $w'_{\{m,s\}} = mssmsw_{\{m,s\}}$.
\item If $j<m<s$ then $w'_{\{m,s\}} = msmsmsw_{\{m,s\}}$.
\end{itemize}
In either case, it is easy to see that $u$-match cannot begin to the left of $w_{\{m,s\}}$, and since $w_{\{m,s\}}$ itself has no $u$-match, $w'$ $u$-represents the edge $ms$.
\end{enumerate}

We have proved that $w'$ $u$-represents $G'$, as desired. \end{proof}

\begin{theorem}\label{thm-1122} Let $u=1^a2^b$, where $a,b\geq 2$ and $a+b=k\geq 3$. Then every graph $G$ is $u$-representable.
\end{theorem}

\begin{proof}  
If $G=K_n$ then by Observation~\ref{observ-Kn} $G$ is $u$-representable. We proceed by induction on the number of edges in a graph with the base case being $K_n$. Our goal is to show that if $G$ is $u$-representable, then the graph $G'$ obtained from $G$ by removing any edge $ij$ ($i<j$) is also $u$-representable.

Suppose that $w$ $u$-represents $G$. We claim that the word
$$w'=u[i,j]12\cdots n12\cdots n w$$
$u$-represents $G'$. 

Indeed, the vertices $i$ and $j$ are not connected any more because $w'_{\{i,j\}}$ contains $u$ (formed by the $k$ leftmost elements of $w'_{\{i,j\}}$). Also, no new edge can be created in $G$ because $w$ is a subword of $w'$. Thus, we only need to show that each edge $ms$ ($m<s$) $u$-represented by $w$ is still $u$-represented by $w'$ if $\{m,s\} \neq \{i,j\}$. We have five cases to consider.

\begin{enumerate}
\item\label{last-thm-121case=m=i=etc} Suppose $m=i$ and $s\neq j$ $(i<s)$. In this case, $w'_{\{i,s\}} = i^{a+1}sisw_{\{i,s\}}$. Because of the subword $sis$, keeping in mind that $a,b\geq 2$, there is no $u$-match in $w'_{\{i,s\}}$ that begins to the left of $w_{\{i,s\}}$, and because $w_{\{i,s\}}$ itself has no $u$-matches, $w'$ $u$-represents the edge $is$.
\item\label{last-thm-121case=s=i=etc} Suppose $s=i$ and $m<i$. In this case, $w'_{\{m,i\}} = i^{a+1}mimiw_{\{m,i\}}$. Because $a,b\geq 2$, there is no $u$-match in $w'_{\{m,i\}}$ that begins to the left of $w_{\{m,i\}}$, and because $w_{\{m,i\}}$ itself has no $u$-matches, $w'$ $u$-represents the edge $mi$.
\item Suppose $m\neq i$ and $s=j$ $(m<j)$. This case is essentially the same as (\ref{last-thm-121case=s=i=etc}) after the substitution $i\rightarrow j$.
\item Suppose $m=j$ and $s>j$. This case is essentially the same as (\ref{last-thm-121case=m=i=etc}) after the substitution $i\rightarrow j$.
\item Suppose $m,s \not \in \{i,j\}$. In this case, $w'_{\{m,s\}} = msmsw_{\{m,s\}}$. Because $a,b\geq 2$, there is no $u$-match in $w'_{\{m,s\}}$ that begins to the left of $w_{\{m,s\}}$, and because $w_{\{m,s\}}$ itself has no $u$-matches, $w'$ $u$-represents the edge $ms$.
\end{enumerate}

We have proved that $w'$ $u$-represents $G'$, as desired. \end{proof}

\section{Concluding remarks}\label{concl-remarks}

We note that the goal of this paper was not in coming up with optimal (shortest possible) $u$-representations of graphs for a given $u$, rather in showing that such representation exists for any $u$ of length at least 3. Thus, we leave it as an open question to improve the upper bounds that can be obtained from our constructions on the length of words $u$-representing graphs for various $u$.

Next step in the theory of graph representations using pattern avoiding words should be in considering ``classical patterns'' rather than ``consecutive patterns'' determining edge/non-edge relations; in the ``classical'' sense letters in a $u$-match does not have to appear next to each other. This direction of research was actually suggested in \cite{JKPR}, while we refer to \cite{Kit} for more information on various types of patterns in words each of which can be used to define edge/non-edge relations. 

\section*{Acknowledgment}
The author is thankful to Bill Chen and Arthur Yang for their great hospitality during the author's stay at the Center of Combinatorics in Nankai University in June-July 2015, when the paper was written.

\end{document}